\date{20 December, 2013}

\documentclass[a4paper,10pt]{amsart}
\usepackage{amsmath}
\usepackage{amsfonts}
\usepackage{amssymb}
\usepackage{amsthm}
\newcommand{\im}{\mathrm{Im}\,}         
\newcommand{\Z}{\mathbb{Z}}

\newcommand{\id}{\mathrm{Id}}   

\newcommand{\la}{\langle}
\newcommand{\ra}{\rangle}
\newcommand{\x}{\times}

\DeclareMathOperator{\ad}{ad}

\DeclareMathOperator{\coker}{coker}

\theoremstyle{plain}
\newtheorem{proposition}{Proposition}

\newtheorem{theorem}[proposition]{Theorem}
\newtheorem{lemma}[proposition]{Lemma}
\newtheorem{corollary}[proposition]{Corollary}
\theoremstyle{definition}
\newtheorem{definition}[proposition]{Definition}

\theoremstyle{remark}
\newtheorem{remark}[proposition]{Remark}

\newcommand{\bR}{\mathbb{R}}

\newcommand{\fg}{\mathfrak{g}}
\newcommand{\fh}{\mathfrak{h}}

\begin{document}

\title{Non-formal Co-symplectic Manifolds}
\date{21 December, 2013}

\author[G. Bazzoni]{Giovanni Bazzoni}
\address{Instituto de Ciencias Matem\'aticas
CSIC-UAM-UC3M-UCM, Consejo Superior de Investigaciones Cient\'{\i}ficas,
C/ Nicolas Cabrera 13-15, 28049 Madrid, Spain}

\email{gbazzoni@icmat.es}

\author[M. Fern\'andez]{Marisa Fern\'{a}ndez}
\address{Universidad del Pa\'{\i}s Vasco,
Facultad de Ciencia y Tecnolog\'{\i}a, Departamento de Matem\'aticas,
Apartado 644, 48080 Bilbao, Spain}

\email{marisa.fernandez@ehu.es}

\author[V. Mu\~{n}oz]{Vicente Mu\~{n}oz}
\address{Facultad de Ciencias Matem\'aticas, Universidad
Complutense de Madrid, Plaza de Ciencias
3, 28040 Madrid, Spain}

\email{vicente.munoz@mat.ucm.es}

\begin{abstract}
We study the formality of the mapping torus of
an orientation-preserving
diffeomorphism of a manifold. In particular, we give
conditions under which a mapping torus has a non-zero Massey product.
As an application we prove that there are non-formal compact co-symplectic
manifolds of dimension $m$ and with first Betti number $b$ if and only if
$m=3$ and $b \geq 2$, or $m \geq 5$ and $b \geq 1$.
Explicit examples for each one of these cases are given.
\end{abstract}

\maketitle

\vskip1pt{\small\textbf{MSC classification [2010]}: Primary 53C15, 55S30;
Secondary 53D35, 55P62, 57R17.}\vskip0.5pt
\vskip1pt{\small\textbf{Key words}: co-symplectic manifold, mapping torus,
minimal model, formal manifold.}\vskip10pt

%
%

\section{Introduction}\label{Introduction}

In this paper we follow the nomenclature of \cite{Li}, where \emph{co-symplectic} manifolds
are the odd-dimensional counterpart to symplectic manifolds. In terms of differential
forms, a \emph{co-symplectic structure} on a $(2n+1)$-dimensional manifold $M$
is determined by a pair $(F,\eta)$ of closed differential forms, where $F$ is  a $2$-form
and $\eta$ is a $1$-form such that $\eta\wedge F^n$ is a volume form, so that $M$
is orientable. In this case, we say that $(M,F,\eta)$ is a \emph{co-symplectic manifold}.
Earlier, such a manifold was called cosymplectic 
 by Libermann \cite{Lib}, or almost-cosymplectic by Goldberg and Yano \cite{GY}.

The simplest examples of co-symplectic manifolds are the manifolds
called \emph{co-K\"ahler} by Li in \cite{Li}, or {cosymplectic} by Blair
\cite{Bl1}. Such a manifold is locally a product of a K\"ahler manifold with a
circle or a line. In fact, a \emph{co-K\"ahler structure} on a $(2n+1)$-dimensional
manifold $M$ is a normal almost contact metric structure $(\phi, \eta, \xi, g)$ on $M$,
that is, a tensor field $\phi$ of type $(1,1)$, a $1$-form
$\eta$, a vector field $\xi$ (the Reeb vector field) with $\eta(\xi)=1$, and a Riemannian
metric $g$ satisfying certain conditions (see section \ref{almost-contact-metric} for details)
such that the $1$-form $\eta$ and the fundamental $2$-form $F$ given by
$F(X,Y)= g(\phi X, Y)$,  for any vector fields $X$ and $Y$ on $M$, are closed.

The topological description of co-symplectic and co-K\"ahler manifolds is due
to Li \cite{Li}. There he proves that a compact manifold $M$ has a co-symplectic
structure if and only if $M$ is the mapping torus of a symplectomorphism
of a symplectic manifold, while $M$ has a co-K\"ahler structure if and only if $M$
is a K\"ahler mapping torus, that is, $M$ is the mapping torus of a Hermitian
isometry on a K\"ahler manifold. This result may be considered an extension
to co-symplectic and co-K\"ahler manifolds of Tischler's Theorem \cite{Tischler} that
asserts that a compact manifold is a mapping torus if and only if it admits a
non-vanishing closed $1$-form.

The existence of a co-K\"ahler structure on a manifold $M$
imposes strong restrictions on the underlying topology of $M$. Indeed, since
co-K\"ahler manifolds are odd-dimensional analogues of K\"ahler manifolds,
several known results from K\"ahler geometry carry over to co-K\"ahler manifolds. In
particular, every compact co-K\"ahler manifold is formal. Another similarity
is the monotone property for the Betti numbers of compact co-K\"ahler manifolds
\cite{CdLM}.

Intuitively, a simply connected manifold is formal if its rational homotopy type
is determined by its rational cohomology algebra. Simply connected compact manifolds
of dimension less than or equal to $6$ are formal \cite{{FM2},{NM}}. We shall say
that $M$ is {\it formal\/} if its minimal model is formal or, equivalently, if the de Rham
complex $(\Omega M,d)$ of $M$ and the algebra of the de Rham cohomology
$(H^*(M),d=0)$ have the same minimal model (see section \ref{Formality-Massey} for details).

It is well known that the existence of a non-zero Massey product is an obstruction
to formality. In \cite{FM2} the concept of formality is extended to a weaker notion
called $s$-formality. There, the second and third authors prove that an orientable
compact connected manifold, of dimension $2n$ or $2n-1$, is formal if and
only if it is $(n-1)$-formal.

The importance of formality in symplectic geometry stems from the fact that it allows to
distinguish symplectic manifolds which admit K\"ahler
structures from those which do not \cite{{CFG}, {FM3}, {OT}}. It seems thus interesting
to analyze what happens for co-symplectic manifolds. In this paper we consider the
following problem on the {\em geography} of co-symplectic manifolds:

\begin{quote}
  For which pairs $(m=2n+1,b)$, with $n,b \geq 1$, are there compact co-symplectic
  manifolds of dimension $m$ and with $b_1=b$ which are non-formal?
\end{quote}

We address this question in section \ref{sec:5}. It will turn out that the answer is the same
as for compact manifolds \cite{FM1}, i.e, that there are always non-formal examples
except for $(m,b)=(3,1)$.

On any compact co-symplectic manifold $M$, the first Betti number must
satisfy $b_1(M)\geq 1$, since the $(2n+1)$-form  $\eta\wedge F^n$ defines a non-zero
cohomology class on $M$, and hence $\eta$ defines a cohomology class $[\eta]\neq 0$.
It is known that any orientable compact manifold of dimension $\leq 4$ and with first Betti
number $b_1=1$ is formal \cite{FM1}.

The main problem in order to answer the question above is to construct examples of non-formal
compact co-symplectic manifolds of dimension $m=3$ with $b_1 \geq 3$
as well as examples of dimension $m=5$ with $b_1=1$. The other cases
are covered in section \ref{sec:5}, using essentially the $3$-dimensional Heisenberg manifold
to obtain non-formal co-symplectic manifolds of dimension $m\geq 3$ and with
$b_1=2$ as well as non-formal co-symplectic manifolds of dimension $m\geq 5$ and with
$b_1\geq 2$,
or from the non-formal compact simply connected symplectic manifold
of dimension $8$ given in \cite{FM3} to exhibit non-formal co-symplectic manifolds
of dimension $m\geq 9$ and with $b_1=1$.

To fill the gaps, we study in section \ref{non-formality-mapping-torus} the
formality of a (not necessarily symplectic) mapping torus $N_{\varphi}$ obtained from $N\times[0,1]$ by identifying $N\times\{0\}$ with $\varphi(N)\times\{1\}$, where $\varphi$ is a self-diffeomorphism of $N$. The description of a minimal
model for a mapping torus can be very complicated even for low degrees. Nevertheless, in
Theorem \ref{thm:minimal-model-mapping-torus} we determine a minimal model
of $N_{\varphi}$ up to some degree $p\geq 2$
when $\varphi$ satisfies some extra conditions, namely that
the map induced on cohomology $\varphi^*: H^k(N) \to H^k(N)$ does not have
the eigenvalue $\lambda=1$, for any $k \leq (p-1)$, but
$\varphi^*: H^p(N) \to H^p(N)$ has the eigenvalue $\lambda=1$ with multiplicity $r\geq 1$.
In particular (see Corollary \ref{mapping-torus-p-formal}), we show that if  $r=1$,
$N_{\varphi}$ is $p$-formal in the sense mentioned above.

Moreover, in Theorem \ref{thm:19} we prove that $N_{\varphi}$ has
a non-zero (triple) Massey product if there exists $p>0$ such that the map
 $$
 \varphi^*:H^p(N)\to H^p(N).
 $$ 
has the eigenvalue $\lambda=1$ with multiplicity 2. In fact, we show that the Massey product  $\la  [dt], [dt],[\tilde\alpha]\ra$
is well-defined on $N_{\varphi}$ and it does not vanish, where
$dt$ is the $1$-form defined on $N_{\varphi}$ by the volume
form on $S^1$, and $[\tilde\alpha]  \in H^p(N_{\varphi})$ is the cohomology class
induced on $N_{\varphi}$ by a certain cohomology class $[\alpha] \in H^p(N)$ fixed by $\varphi^*$.

Regarding symplectic mapping torus manifolds, first we
notice that if $N$ is a compact symplectic $2n$-manifold, and
$\varphi:N\to N$ is a symplectomorphism, then the map induced on cohomology
$\varphi^*: H^2(N) \to H^2(N)$ always has the eigenvalue $\lambda=1$.
As a consequence of Theorem \ref{thm:minimal-model-mapping-torus}, we get that
 if $N_{\varphi}$ is a symplectic mapping torus such that the map
 $\varphi^*: H^1(N) \to H^1(N)$ does not have the eigenvalue
$\lambda=1$, then $N_{\varphi}$ is $2$-formal if and only if
the eigenvalue $\lambda=1$ of $\varphi^*: H^2(N) \to H^2(N)$ has multiplicity $r=1$.
Thus, in these conditions, the co-symplectic manifold $N_{\varphi}$
is formal when $N$ has dimension four.

In section \ref{sec:5}, using Theorem \ref{thm:19}, we solve the case $m=3$ with $b_1\geq 3$
taking the mapping torus of a symplectomorphism of  a surface of genus $k\geq 2$
(see Proposition  \ref{m=3}).
For  $m=5$ and $b_1=1$ we consider the mapping torus of a
symplectomorphism of a 4-torus (see Proposition  \ref{prop:3}).

Let $G$ be a connected, simply connected solvable Lie group,
and let $\Gamma\subset G$ be a discrete, cocompact subgroup.
Then $M=\Gamma\backslash G$ is a solvmanifold. The manifold
constructed in Proposition \ref{prop:3} is not
a solvmanifold according to our definition. However, it is the quotient
of a solvable Lie group by a closed subgroup. In section \ref{sec:6} we
present an explicit example of a non-formal compact
co-symplectic $5$-dimensional manifold $S$, with first Betti number
$b_1(S)=1$, which is a solvmanifold. We describe $S$ as the mapping torus
of a symplectomorphism of a $4$-torus, so
this example fits in the scope of Proposition \ref{prop:3}.

%
%

\section{Minimal models and formality} \label{Formality-Massey}

In this section we
recall some fundamental facts of the theory of minimal models. For more details, see \cite{DGMS}, \cite{FHT} and \cite{FOT}.

We work over the field $\bR$ of real numbers. Recall that a \textit{commutative differential graded algebra} $(A,d)$ (CDGA for short) is a graded algebra $A=\oplus_{k\geq 0}A^k$ which is graded commutative, i.e. $x\cdot y=(-1)^{|x||y|}y\cdot x$ for homogeneous elements $x$ and $y$, together with a differential $d:A^k\to A^{k+1}$ such that $d^2=0$ and
$d(x\cdot y)=dx\cdot y+(-1)^{|x|}x\cdot dy$ (here $|x|$ denotes the degree of the homogeneous element $x$).

Morphisms of CDGAs are required to preserve the degree and to commute with the differential. Notice that the cohomology of a CDGA is an algebra which can be turned into a CDGA by endowing it with the
zero differential. A CDGA is said to be \textit{connected} if $H^0(A,d)\cong\bR$. The main example of CDGA
is the de Rham complex of a smooth manifold $M$, $(\Omega^*(M),d)$, where $d$ is the exterior differential.

A CDGA $(A,d)$ is said to be \textit{minimal} (or \textit{Sullivan}) if the following happens:
\begin{itemize}
\item $A=\bigwedge V$ is the free commutative algebra generated by a graded (real) vector space $V=\oplus_kV^k$;
\item there exists a basis $\{x_i, \ i\in\mathcal{J}\}$ of $V$, for a well-ordered index set $\mathcal{J}$, such that
$|x_i|\leq\ |x_j|$ if $i<j$ and the differential of a generator $x_j$ is expressed in terms of the preceding $x_i$ ($i<j$); in particular, $dx_j$ does not have a linear part.
\end{itemize}

We have the following fundamental result:
\begin{proposition}
Every connected CDGA $(A,d)$ has a minimal model, that is, there exist a
minimal algebra $(\bigwedge V,d)$ together with a morphism of CDGAs $\varphi:(\bigwedge V,d)\to (A,d)$ which induces an isomorphism $\varphi^*:H^*(\bigwedge V,d)\to
H^*(A,d)$. The minimal model is unique.
\end{proposition}

The (real) \textit{minimal model} of a differentiable manifold $M$ is by definition
the minimal model of its de Rham algebra $(\Omega^*(M),d)$.

Recall that a minimal algebra $(\bigwedge V,d)$ is {\it formal} if there exists a
morphism of differential algebras $\psi\colon {(\bigwedge V,d)}\longrightarrow
(H^*(\bigwedge V),0)$ that induces the identity on cohomology.
Also a differentiable manifold $M$ is {\it formal\/} if its minimal model is
formal. Many examples of formal manifolds are known: spheres, projective
spaces, compact Lie groups, homogeneous spaces, flag manifolds,
and compact K\"ahler manifolds.

In \cite{DGMS}, the formality of a minimal algebra is characterized as follows.

\begin{proposition} \label{prop:criterio1}
A minimal algebra $(\bigwedge V,d)$ is formal if and only if the space $V$ can be
decomposed as a direct sum $V= C\oplus
N$ with $d(C) = 0$, $d$ injective on $N$ and such that every
closed element in the ideal $I(N)$ generated by $N$ in $\bigwedge
V$ is exact.
\end{proposition}

This characterization of formality can be weakened using the concept of
$s$-formality introduced in \cite{FM2}.

\begin{definition}\label{def:primera}
A minimal algebra $(\bigwedge V,d)$ is $s$-formal
($s> 0$) if for each $i\leq s$
the space $V^i$ of generators of degree $i$ decomposes as a direct
sum $V^i=C^i\oplus N^i$, where the spaces $C^i$ and $N^i$ satisfy
the three following conditions:
\begin{enumerate}

\item $d(C^i) = 0$,

\item the differential map $d\colon N^i\rightarrow \bigwedge V$ is
injective,

\item any closed element in the ideal
$I_s=I(\bigoplus\limits_{i\leq s} N^i)$, generated by the space
$\bigoplus\limits_{i\leq s} N^i$ in the free algebra $\bigwedge
(\bigoplus\limits_{i\leq s} V^i)$, is exact in $\bigwedge V$.

\end{enumerate}

\end{definition}

A smooth manifold $M$ is $s$-formal if its minimal model
is $s$-formal. Clearly, if $M$ is formal then $M$ is $s$-formal, for any $s>0$.
The main result of \cite{FM2} shows that sometimes the weaker
condition of $s$-formality implies formality.

\begin{theorem} \label{fm2:criterio2}
Let $M$ be a connected and orientable compact differentiable
manifold of dimension $2n$, or $(2n-1)$. Then $M$ is formal if and
only if is $(n-1)$-formal.
\end{theorem}

In order to detect non-formality, instead of computing the minimal
model, which usually is a lengthy process, we can use Massey
products, which are obstructions to formality. Let us recall their
definition. The simplest type of Massey product is the triple
Massey product. Let $(A,d)$ be a CDGA and suppose $a,b,c\in H^*(A)$ are
three cohomology classes such that $a\cdot b=b\cdot c=0$.
Take cocycles $x$, $y$ and $z$ representing these cohomology classes
and let $s$, $t$ be elements of $A$ such that
$$
ds=(-1)^{|x|}x\cdot y, \quad dt=(-1)^{|y|}y\cdot z.
$$
Then one checks that
$$
w=(-1)^{|x|}x\cdot t+(-1)^{|x|+|y|-1}s\cdot z
$$
is a cocyle. The choice of different representatives gives an indeterminacy,
represented by the space
$$
\mathcal{I}=a\cdot H^{|y|+|z|-1}(A)+H^{|x|+|y|-1}(A)\cdot c.
$$
We denote by $\langle a,b,c\rangle$ the image of the cocycle $w$ in
$H^*(A)/\mathcal{I}$. As is proven in \cite{DGMS} (and which is essentially
equivalent to Proposition \ref{prop:criterio1}),
if a minimal CDGA is formal, then one can make uniform choices of cocyles
so that the classes representing (triple) Massey products are exact.
In particular, if the real minimal model of a
manifold contains a non-zero Massey product, then the manifold is not formal.

%
%
%

\section{Co-symplectic manifolds}\label{almost-contact-metric}

In this section we recall some definitions and results about co-symplectic
manifolds, and we extend to co-symplectic Lie algebras the result
of Fino-Vezzoni \cite{FV} for co-K\"ahler Lie algebras.

\begin{definition}
Let $M$ be a $(2n+1)$-dimensional manifold. An \emph{almost contact metric structure} on $M$ consists
of a quadruplet $(\phi,\xi,\eta,g)$, where
$\phi$ is an endomorphism of the tangent bundle $TM$, $\xi$  is a vector field,  $\eta$ is
a $1$-form and  $g$ is a Riemannian metric on $M$ satisfying the conditions
 \begin{equation} \label{eqn:almost-contact}
\phi^2 = -\id + \eta \otimes \xi, \quad  \quad  \eta(\xi)=1, \quad  \quad
g(\phi X, \phi Y) = g(X,Y) - \eta(X) \eta(Y),
\end{equation}
for $X,Y\in\Gamma(TM)$.
\end{definition}

Thus, $\phi$ maps the distribution $\ker(\eta)$ to itself and satisfies $\phi(\xi)=0$. We call
$(M,\phi,\eta,\xi,g)$ an \textit{almost contact metric} manifold. The \emph{fundamental} $2$-form $F$ on $M$ is defined by
$$
F(X,Y) = g(\phi X, Y),
$$
for $X,Y\in\Gamma(TM)$.

Therefore, if $(\phi,\xi,\eta,g)$ is an almost contact metric structure on $M$
with fundamental $2$-form $F$, then $\eta \wedge F^n\not=0$ everywhere.
Conversely (see \cite{Bl3}), if $M$ is a differentiable manifold of dimension $2n+1$
with a $2$-form $F$ and a $1$-form $\eta$ such that $\eta \wedge F^n$
is a volume form on $M$,  then there exists an almost contact metric structure
$(\phi,\xi, \eta, g)$ on $M$ having $F$ as the fundamental form.

There are different classes of structures that can be considered on $M$ in terms
of $F$ and $\eta$ and their covariant derivatives. We recall here those
that are needed in the present paper:

\begin{itemize}
\item $M$ is \emph{co-symplectic} iff $dF = d\eta=0$;
\item $M$ is \emph{normal} iff the Nijenhuis torsion $N_{\phi}$ satisfies $N_{\phi}=-2 d\eta\otimes \xi$;
\item $M$ is \emph{co-K\"ahler} iff it is normal and co-symplectic or, equivalently, $\phi$ is parallel,
\end{itemize}
where the Nijenhuis torsion $N_{\phi}$ is given by
$$
N_{\phi}(X,Y) = \phi^2[X,Y] + [\phi X, \phi Y] - \phi [\phi X, Y] - \phi [X, \phi Y],
$$
for $X,Y\in\Gamma(TM)$.

In the literature, co-symplectic manifolds are often called
\textit{almost cosymplectic}, while co-K\"ahler manifolds are called \textit{cosymplectic}
(see \cite{Bl1,Bl2,CdLM,FV}).

Let us recall that a \emph{symplectic} manifold $(M,\omega)$
is a pair consisting of a $2n$-dimensional differentiable manifold $M$ with a closed
$2$-form $\omega$ which is non-degenerate (that is, $\omega^n$
never vanishes). The form $\omega$ is called symplectic. The following well known result shows that co-symplectic manifolds are really the odd dimensional analogue of
symplectic manifolds; a proof can be found in Proposition 1 of \cite {Li}.

\begin{proposition}
A manifold $M$ admits a co-symplectic structure if and only
 if the product $M\times S^1$ admits an $S^1$-invariant symplectic form.
\end{proposition}

A theorem by Tischler \cite{Tischler} asserts that a compact manifold is
a mapping torus if and only if it admits a non-vanishing closed $1$-form.
This result was extended recently
to co-symplectic manifolds by Li \cite{Li}. Let us recall first some definitions.

Let $N$ be a differentiable manifold and let $\varphi:N\to N$ be a diffeomorphism.
The {\em mapping torus} $N_{\varphi}$ of $\varphi$ is the
manifold obtained from $N\times[0,1]$ by identifying the ends with $\varphi$, that is
$$
N_{\varphi}= \frac {N\times[0,1]}{(x,0) \sim (\varphi(x),1)}.
$$
It is a differentiable manifold, because it is the quotient of $N \times \bR$
by the infinite cyclic group generated by $(x,t) \to (\varphi(x),t+1)$. The natural
map $\pi \colon N_{\varphi} \to S^1$ defined by $\pi(x,t)=e^{2\pi it}$ is the projection
of a locally trivial fiber bundle.

\begin{definition}
Let $N_{\varphi}$ be a mapping torus of a diffeomorphism $\varphi$ of $N$.
We say that $N_{\varphi}$ is a \textit{symplectic} mapping torus
if  $(N,\omega)$ is a symplectic manifold and $\varphi:N\to N$ a symplectomorphism, that is,
$\varphi^*\omega=\omega$.
\end{definition}

\begin{theorem} [Theorem 1, \cite {Li}]
A compact manifold $M$ admits a co-symplectic structure if and only if it is a symplectic
mapping torus $M=N_{\varphi}$.
\end{theorem}

Notice that if $M$ is a symplectic mapping torus $M=N_{\varphi}$, then the pair $(F,\eta)$ defines
a co-symplectic structure on $M$, where $F$ is the closed $2$-form on $M$ defined by the symplectic
form on $N$, and
$$
\eta=\pi^*(\theta),
$$
with $\theta$ the volume form on $S^1$. Moreover,
notice that any $3$-dimensional mapping torus is  a symplectic mapping torus
if the corresponding diffeomorphism preserves the orientation, since
such a diffeomorphism is isotopic to an area preserving one.
However, in higher dimensions, there exist mapping tori without co-symplectic
structures. That is, they are not symplectic mapping tori (see Remark \ref{remark:4} in section \ref{sec:5} and \cite{Li}).

Next, we consider a Lie algebra $\fg$ of dimension $2n+1$ with
an \textit{almost contact metric} structure, that is, with a quadruplet
$(\phi,\xi,\eta,g)$ where $\phi$ is an endomorphism of $\fg$,
$\xi$ is a non-zero vector in $\fg$, $\eta\in\fg^*$ and $g$ is a scalar
product in $\fg$, satisfying
(\ref{eqn:almost-contact}). Then, $\fg$ is said to be \emph{co-symplectic}
iff $dF = d\eta=0$; and $\fg$ is called  \emph{co-K\"ahler} iff it is normal
and co-symplectic, where  $d:\wedge^k\fg^*\to\wedge^{k+1}\fg^*$ is the
Chevalley-Eilenberg differential.

The following result is proved in \cite{FV}.
\begin{proposition}
Co-K\"ahler Lie algebras in dimension $2n+1$ are in one-to-one correspondence
with $2n$-dimensional K\"{a}hler Lie algebras endowed with a skew-adjoint
derivation $D$ which commutes with its complex structure.
\end{proposition}

In order to extend this correspondence to co-symplectic Lie algebras we need
to recall the following. Let $(V,\omega)$ be a symplectic vector space (hence $\omega$ is a skew-symmetric invertible matrix).
An element $A\in\mathfrak{gl}(V)$ is an infinitesimal symplectic
transformation if $A\in\mathfrak{sp}(V)$, that is, if
$$
A^t\omega+\omega A=0.
$$
A scalar product $g$ on $(V,\omega)$ is said to be compatible
with $\omega$ if the endomorphism $J:V\to V$ defined by
$\omega(u,v)=g(u,J v)$ satisfies $J^2=-\id$.
We prove the following:

\begin{proposition}\label{prop:107}
Co-symplectic Lie algebras of dimension $2n+1$ are in one-to-one correspondence with $2n$-dimensional symplectic Lie
algebras endowed with a compatible metric and a derivation $D$ which is an infinitesimal symplectic transformation.
\end{proposition}

\begin{proof}
Let $(\phi,\xi,\eta,g)$ be a co-symplectic structure on a Lie algebra $\fg$ of dimension $2n+1$.
Set $\fh=\ker(\eta)$. For $u,v\in\fh$ we compute
$$
\eta([u,v])=-d\eta(u,v)=0,
$$
since $\eta$ is closed (this is simply Cartan's formula applied to the case in which $\eta(u)$ and $\eta(v)$ are constant). Then $\fh$ is a Lie subalgebra of $\fg$.
Note that $\fh$ inherits an almost complex structure $J$ and a metric $g$ which are compatible.
From $\phi$ and $g$ we obtain the $2$-form $\omega$ which is closed and non-degenerate by hypothesis.
Thus $(\fh,\omega)$ is a symplectic Lie algebra.

Actually $\fh$ is an ideal of $\fg$. Indeed, the fact that $\eta(\xi)= 1$ implies
that $\xi$ does not belong to $[\fg,\fg]$, and then one has
$$
[\fh,\fh]\subseteq \fh \quad \textrm{and} \quad [\xi,\fh]\subseteq\fh.
$$
Thus one can write
$$
\fg=\mathbb{R}\xi\oplus_{\ad_{\xi}}\fh.
$$
Since $\omega$ is closed, we obtain
\begin{align}\label{eq:156}
0&=d\omega(\xi,u,v)=-\omega([\xi,u],v)+\omega([u,v],\xi)-\omega([v,\xi],u)= \nonumber \\
&=-\omega(\ad_{\xi}(u),v)-\omega(u,\ad_{\xi}(v)).
\end{align}
The correspondence $X\mapsto \ad_{\xi}(X)$ gives a derivation $D$ of $\fh$ (this follows from the Jacobi identity in $\fg$) and the above equality shows that $D$ is an infinitesimal symplectic transformation.

Next suppose we are given a symplectic Lie algebra $(\fh,\omega)$ endowed with a metric $g$ and a derivation $D\in\mathfrak{sp}(\fh)$. Set
$$
\fg=\mathbb{R}\xi\oplus\fh
$$
and define the following Lie algebra structure on $\fg$:
$$
[u,v]:=[u,v]_{\fh}, \quad [\xi,u]:=D(u), \quad u,v\in\fh.
$$
Since $D$ is a derivation of $\fh$, the Jacobi identity holds in $\fg$. Let $J$ denote the almost complex structure compatible with $\omega$ and $g$. Extend $J$ to an endomorphism $\phi$ of $\fg$ setting $\phi(\xi)=0$ and extend $g$ so that $\xi$ has length 1 and $\xi$ is orthogonal to $\fh$. Also, let $\eta$ be the dual $1$-form with respect to the metric $g$. It is immediate to see that $d\eta=0$. On the other hand, equation (\ref{eq:156}) shows that $d\omega=0$ as $D$ is an infinitesimal symplectic transformation. Thus $\fg$ is a co-symplectic Lie algebra.
\end{proof}

\begin{remark}
If one wants to obtain a co-symplectic \textit{nilpotent} Lie algebra, then the initial data in
Proposition \ref{prop:107} must be modified so that the symplectic Lie algebra and the derivation $D$ are nilpotent.
This gives a way to classify co-symplectic nilpotent Lie algebras in dimension $2n+1$ starting from nilpotent symplectic Lie algebras in dimension $2n$ and a nilpotent symplectic derivation.
\end{remark}

%
%

\section{Minimal models of mapping tori} \label{non-formality-mapping-torus}

In this section we study the formality of the mapping torus
of a orientation-preserving diffeomorphism of a manifold.
We start with some useful results.

\begin{lemma}\label{lem:18'}
Let $N$ be a smooth manifold and let $\varphi:N\to N$ be a diffeomorphism.
Let $M=N_{\varphi}$ denote the mapping torus of $\varphi$.
Then the cohomology of $M$ sits in an exact sequence
 $$
 0 \to C^{p-1} \to H^p(M) \to K^{p}\to 0,
 $$
where $K^p$ is the kernel of $\varphi^*-\id\colon H^{p}(N)\to  H^{p}(N)$, and $C^{p}$ is its cokernel.
\end{lemma}

\begin{proof}
 This is a simple application of the Mayer-Vietoris sequence. Take $U,V$ two
 open intervals
 covering $S^1=[0,1]/0 \sim 1$, where $U\cap V$ is the disjoint union of two intervals.
 Let $U'=\pi^{-1}(U)$, $V'=\pi^{-1}(V)$. Then $H^p(U')\cong H^p(N)$,
 $H^p(V') \cong H^p(N)$ and $H^p(U'\cap V')\cong H^p(N)\oplus H^p(N)$. The
Mayer-Vietoris sequence associated to this covering becomes
 \begin{align}\label{eqn:MV}
  \ldots &\to H^p(M) \to H^p(N) \oplus H^p(N)  \stackrel{F}{\longrightarrow} H^p(N) \oplus H^p(N)
 \to H^{p+1}(M) \to  \nonumber\\
  &\to H^{p+1}(N)\oplus H^{p+1}(N)  \to \ldots
 \end{align}
 where the map $F$ is $([\alpha],[\beta]) \mapsto ([\alpha]-[\beta],[\alpha]-\varphi^*[\beta])$.

 Write
 $$
 K=\ker \Big(\varphi^*-\id: H^*(N)\to H^*(N)\Big), \quad \text{and}  \quad
 C=\coker  \Big(\varphi^*-\id: H^*(N)\to H^*(N) \Big).
 $$
 These are graded vector spaces $K=\bigoplus K^p$, $C=\bigoplus C^p$.
 The exact sequence (\ref{eqn:MV}) then yields an exact sequence
 $0 \to C^{p-1} \to H^p(M) \to K^{p}\to 0$.
\end{proof}

Let us look more closely at the exact sequence in Lemma \ref{lem:18'}.
First take $[\beta]\in C^{p-1}$. Then $[\beta]$ can be thought as
an element in $H^{p-1}(N)$
modulo $\im (\varphi^*-\id)$. The map $C^{p-1} \to H^p(M)$ in Lemma
\ref{lem:18'} is the connecting homomorphism $\delta^*$. This is worked
out as follows (see \cite{BottTu}):
take a smooth function $\rho(t)$ on $U$ which equals $1$ in one
of the intervals of $U\cap V$ and zero on the other. Then
  \begin{equation}\label{eqn:referee}
 \delta^* [\beta]= [d\rho\wedge \beta].
  \end{equation}
Write $\tilde\beta=d\rho\wedge \beta$. If we put the point $t=0$ in $U\cap V$, then clearly
$\tilde\beta(x,0)= \tilde\beta(x,1)=0$, so $\tilde\beta$ is a well-defined closed $p$-form
on $M$. (Note that $[d\rho]=[\eta] \in H^1(S^1)$, where
$\eta=\pi^*(\theta)=dt$, so $[\tilde\beta]\in H^p(M)$ is
$[\eta \wedge \beta]$.)

On the other hand, if $[\alpha]\in K^p$, then $\varphi^*[\alpha]=[\alpha]$.
So $\varphi^*\alpha=\alpha+d\theta$, for some $(p-1)$-form $\theta$. Let us take
a function $\rho:[0,1]\to[0,1]$
such that $\rho\equiv 0$ near $t=0$ and $\rho\equiv 1$ near $t=1$.
Then, the closed $p$-form $\tilde\alpha$ on $N\times[0,1]$ given by
 \begin{equation} \label{eqn:tilde-alpha}
    \tilde{\alpha}(x,t)=\alpha(x)+d(\rho(t)\theta(x)),
 \end{equation}
where $x\in N$ and $t\in [0,1]$,
defines a closed $p$-form $\tilde\alpha$ on $M$. Indeed,
$\varphi^*\tilde{\alpha}(x,0)=\varphi^*\alpha=\alpha+d\theta=\tilde{\alpha}(x,1)$.
Moreover, the class
$[\tilde\alpha]\in H^p(M)$ restricts to $[\alpha]\in H^p(N)$.
This gives a splitting
  $$
  H^p(M) \cong C^{p-1} \oplus K^{p}.
 $$

\begin{theorem}\label{thm:19}
Let $N$ be an oriented compact
smooth manifold of dimension $n$, and let
$\varphi:N\to N$ be an orientation-preserving diffeomorphism.
Let $M=N_{\varphi}$ be the mapping torus of $\varphi$. Suppose that,
for some $p>0$, the homomorphism $\varphi^*:H^p(N)\to H^p(N)$ has eigenvalue $\lambda=1$ with multiplicity\footnote{In this paper, by \textit{multiplicity} of the eigenvalue $\lambda$ of an endomorphism $A:V\to V$ we mean the multiplicity of $\lambda$ as a root of the minimal polynomial of $A$.} two.
Then $M$ is non-formal since there exists a non-zero (triple) Massey
product. More precisely,
if $[\alpha]  \in K^p \subset H^p(N)$ is such that
 $$
 [\alpha] \in \im \Big(\varphi^*-\id: H^p(N)\to H^p(N)\Big),
 $$
then the Massey product $\la  [\eta], [\eta],[\tilde\alpha]\ra$ does not vanish.
\end{theorem}

\begin{proof}
First, we notice that if the eigenvalue $\lambda=1$ of
 $ \varphi^*:H^p(N)\to H^p(N)$
has multiplicity two, then there exists $[\alpha]\in H^p(N)$ satisfying the conditions
mentioned in Theorem \ref{thm:19}. In fact, denote by
$$
 E=\ker \, (\varphi^{*}-\id)^2
 $$
 the graded eigenspace corresponding to $\lambda=1$. Then
$K=\ker(\varphi^{*}-\id) \subset E$ is a proper subspace. Take
 \begin{equation} \label{eqn:pairing}
[\beta]\in E^p \setminus K^p \subset H^p(N) \quad  \text{and}  \quad  [\alpha] = \varphi^*[\beta]-[\beta].
 \end{equation}

Thus $[\alpha]\in K^p \cap \im  \Big(\varphi^*-\id: H^p(N)\to H^p(N) \Big)$. By (\ref{eqn:referee})
and Lemma \ref{lem:18'}, 
the Massey product $\la  [\eta], [\eta],[\tilde\alpha]\ra$ is well-defined.
In order to prove that it is non-zero we proceed as follows. Clearly,
 $$
 C\cong E/I, \quad  \text{where } \ I =\im (\varphi^{*} -\id) \cap E.
 $$

As $\varphi$ is an orientation-preserving diffeomorphism, the Poincar\'e duality pairing
satisfies that $\la \varphi^*(u),\varphi^*(v)\ra=
\la u,v\ra$, for $u\in H^p(N)$, $v\in H^{n-p}(N)$.
Therefore the $\lambda$-eigenspace
of $\varphi^*$, $E_\lambda$, pairs
non-trivially only with $E_{1/\lambda}$. In particular, Poincar\'e duality gives
a perfect pairing
 $$
 E^p \times E^{n-p} \to \bR.
 $$
Now $K^p \times I^{n-p}$ is sent to
zero: if $x\in \ker(\varphi^*-\id)$ and $y=\varphi^*(z)-z$, then $\langle
x,y\rangle= \langle x,\varphi^*(z)-z\rangle=
\langle x,\varphi^*(z) \rangle - \langle x,z \rangle=
\langle \varphi^*(x),\varphi^*(z) \rangle - \langle x,z \rangle=0$.
Therefore there is a perfect pairing
$$
E^p/K^p \times I^{n-p} \to \bR.
$$

Take  $[\beta]$ and $[\alpha]$ as in (\ref{eqn:pairing}).
By the discussion above about Poincar\'e duality, there is some
$[\xi]\in I^{n-p}$ such that
 $$
 \la [\beta],[\xi]\ra \neq 0.
 $$
Note that in particular, $[\xi]$ pairs trivially with all elements in $K^p$.

Consider now the form $\tilde\alpha$ on $M$ corresponding to $\alpha$ as in (\ref{eqn:tilde-alpha}),
$[\tilde\alpha]\in H^p(M)$.
Let us take the $p$-form $\gamma$ on $N$ defined by
$$
\gamma= \int_{0}^{1} \tilde{\alpha}(x,s)ds.
$$
Then $[\gamma]=[\alpha]=\varphi^*[\beta]-[\beta]$ on $N$. Hence we can write
$$
\gamma=\varphi^*\beta -\beta + d\sigma,
$$
for some $(p-1)$-form $\sigma$ on $N$. Now let us set
$$
\tilde{\gamma}(x,t)= \left(\int_{0}^{t} \tilde{\alpha}(x,s)ds\right) + \beta + d(\zeta(t) (\varphi^*)^{-1}\sigma),
$$
where $\zeta(t)$, $t\in [0,1]$, equals $1$ near $t=0$, and equals $0$ near $t=1$. Then
$$
\varphi^* (\tilde{\gamma}(x,0)) =
\varphi^*(\beta + d((\varphi^*)^{-1}\sigma))=
\varphi^*\beta + d\sigma=\gamma+\beta=\tilde{\gamma}(x,1),
$$
so $\tilde{\gamma}$ is a well-defined $p$-form on $M$. Moreover,
 $$
 d(\tilde{\gamma}(x,t)) = dt \wedge \tilde{\alpha}(x,t)
 $$
on the mapping torus $M$.
Therefore we have the Massey product
 \begin{equation} \label{eqn:MProduct}
 \la [dt],[dt], [\tilde\alpha] \ra= [dt \wedge \tilde{\gamma}].
 \end{equation}

We need to see that this Massey product is non-zero.
For this, we multiply against $[\tilde\xi]$, where $\tilde{\xi}$ is the $(n-p)$-form
on $M$ associated to $\xi$ by the formula (\ref{eqn:tilde-alpha}). 
Recall that $[\xi]\in I^{n-p}\subset K^{n-p}
\subset H^{n-p}(M)$. We have
 $$
 \langle [dt \wedge \tilde{\gamma}], [\tilde\xi]\rangle =\int_M dt \wedge \tilde{\gamma}\wedge \tilde\xi=
 \int_0^1 \left( \int_{N\x \{t\}} \tilde{\gamma} \wedge \tilde\xi \right) dt\, .
 $$
Restricting to the fibers, we have $[\tilde{\gamma}|_{N\x\{t\}}] =t[\alpha]+[\beta]$
and $[\tilde\xi|_{N\x\{t\}}] =[\xi]$. Moreover,
$\la [\alpha],[\xi]\ra=0$ and
$\la [\beta],[\xi]\ra =\kappa \neq 0$.
So $\int_{N\x \{t\}} \tilde{\gamma} \wedge \tilde\xi= \kappa \neq 0$. Therefore
 $$
 \langle [dt \wedge \tilde{\gamma}], [\tilde\xi]\rangle =\kappa\neq 0\, .
 $$

Now the indeterminacy of the Massey product is in the space
 $$
 {\mathcal I}= [\tilde\alpha] \wedge H^{1}(M) + [\eta] \wedge H^{p}(M).
 $$
To see that the Massey product (\ref{eqn:MProduct}) does not live in ${\mathcal I}$, it is
enough to see that the elements in ${\mathcal I}$ pair trivially with $[\tilde{\xi}]$.
On the one hand, $\tilde\alpha\wedge\tilde{\xi}$ is exact in every fiber (since
$\langle [\alpha],[\xi]\rangle=0$ on $N$). Therefore $[\tilde\alpha]\wedge[\tilde{\xi}]=0$.
On the other hand, $H^p(M)\cong C^{p-1} \oplus K^p$. The elements corresponding to
$C^{p-1}$ all have a $dt$-factor. Hence the elements in $[\eta] \wedge H^{p}(M)$
are of the form $[dt\wedge \tilde\delta]$, for some $[\delta]\in K^p\subset H^p(N)$.
But then $\langle [dt\wedge \tilde\delta],[\tilde\xi]\rangle=
\int_M dt\wedge \tilde\delta\wedge\tilde\xi= \langle [\delta],[\xi]\rangle=0$.
\end{proof}

\begin{remark} \label{rem:otro}
The non-formality of the mapping torus $M$
is proved in \cite[Proposition 9]{FGM}
when $p=1$ and the eigenvalue $\lambda=1$ has multiplicity
$r\geq 2$, by a different method.
 \end{remark}

We finish this section with the following result,  which gives a
partial computation of the minimal model of $M$.

From now on we write
 $$
 \varphi^*_k : H^k(N) \to H^k(N),
 $$
for each $1\leq k\leq n$,
the induced morphism on cohomology by a diffeomorphism
$\varphi: N\to N$.

\begin{theorem} \label{thm:minimal-model-mapping-torus}
With $M=N_\varphi$ as above, suppose that there is some $p\geq 2$ such
that $\varphi^*_k$ does not have the eigenvalue
$\lambda=1$ (i.e. $\varphi^*_k -\id$
is invertible) for any $k \leq (p-1)$, and that $\varphi^*_p$ does have 
the eigenvalue $\lambda=1$ with some multiplicity 
$r\geq1$. Denote
  $$
 K_j= \ker \Big((\varphi_{p}^*- \id)^j: H^p(N) \to H^p(N)\Big),
  $$
for $j=0,\ldots, r$. So $\{0\}=K_0\subset K_1 \subset K_2 \subset \ldots \subset K_r$.
Write $G_j=K_j/K_{j-1}$,
$j=1,\ldots, r$. The map $F=\varphi_{p}^*-\id$ induces maps
$F: G_j \to G_{j-1}$, $j=1,\ldots, r$ (here $G_0=0$).

 Then the minimal model of $M$ is, up to degree $p$, given by the following generators:
  \begin{align*}
  W^1 &= \la a \ra, \qquad da=0, \\
  W^k &= 0, \qquad k = 2, \ldots, p-1, \\
  W^p &= G_1 \oplus G_2 \oplus \ldots \oplus G_r, \qquad dw= a \cdot F(w), \, w\in G_j.
  \end{align*}
 \end{theorem}

 \begin{proof}
 We need to construct a map of differential algebras
  $$
  \rho: (\bigwedge (W^1 \oplus W^p), d) \to (\Omega^*(M),d)
  $$
 which induces an isomorphism in cohomology up to degree $p$ and an injection in degree $p+1$
 (see \cite{DGMS}). By Lemma \ref{lem:18'}, we have that
 \begin{align*}
  H^1(M) &=\la [dt]\ra, \\
H^k(M)& =0, \qquad 2\leq k\leq p-1, \\
  H^p(M) &=\ker (\varphi^*_p-\id) =K_1, \\
  H^{p+1}(M) &= \left( [dt]\wedge \coker (\varphi^*_p-\id)\right) \oplus \ker (\varphi^*_{p+1}-\id)
\end{align*}

We start by setting $\rho(a)=dt$, where $t$ is the coordinate of $[0,1]$ in the
description
  $$
  M=(N\x [0,1])/(x,0)\sim (\varphi(x),1)\,.
  $$
This automatically gives that $\rho$ induces an isomorphism in cohomology up to degree $p-1$.
Now let us go to degree $p$. Take a Jordan block of $\varphi^*_p$ for the
eigenvalue $\lambda=1$. Let $1\leq j_0\leq r$ be its size. Then we may take $v\in K_{j_0}\setminus
K_{j_0-1}$ in it. First, this implies that $v\notin I=\im (\varphi^*_p-\id)$. Set
 $$
 v_j=(\varphi^*_p-\id)^{j_0-j} v \in K_j\, ,
 $$
for $j=1,\ldots, j_0$. Now let $b_j$ denote the class of $v_j$ on $G_j=K_j/K_{j-1}$. Then
$d(b_j)=a\cdot b_{j-1}$. We want to define $\rho$ on $b_1,\ldots, b_{j_0}$. For
this, we need to construct forms $\tilde\alpha_1,\ldots,\tilde\alpha_{j_0}\in \Omega^p(M)$
such that $[\tilde\alpha_1]$ represents $v_1\in K_1=H^p(M)$, and
 $$
 d\tilde\alpha_j= dt \wedge \tilde\alpha_{j-1}\, .
 $$
Then we set $\rho(b_j)=\tilde\alpha_j$, and $\rho$ is a map of differential algebras.

We work inductively. Let $v_j=[\alpha_j]\in H^p(N)$. Here $\varphi^*[\alpha_j]-[\alpha_j]=
[\alpha_{j-1}]$. As $\varphi^*[\alpha_1]-[\alpha_1]=0$, we have that
$\varphi^*\alpha_1 =\alpha_1+ d\theta_1$. Set
  $$
  \tilde\alpha_1(x,t)= \alpha_1(x) + d(\zeta(t)\theta_1(x)),
 $$
where $\zeta: [0, 1] \to [0, 1]$ is a smooth function such that
$\zeta \equiv 0$ near $t = 0$ and $\zeta\equiv 1$ near $t = 1$.
Clearly, $[\tilde\alpha_1]=[\alpha_1]=v_1$.

Assume by induction that $\tilde\alpha_1,\ldots, \tilde\alpha_j$ have been constructed,
and moreover satisfying that
 $$
 [\tilde{\alpha}_{k}|_{N\x\{t\}}] = [\alpha_{k}] +
 \sum_{i=1}^{k-1} c_{ik}(t) [\alpha_i],
 $$
for some polynomials $c_{ik}(t)$, $k=1,\ldots, j$.
Note that the result holds for $k=1$.
To construct $\tilde\alpha_{j+1}$, we work as follows. We define
 $$
  \gamma_j(x)= \int_0^1 \left(\tilde\alpha_j - \sum_{i=1}^{j-1} c_i \tilde\alpha_i\right) dt
 $$
This is a closed form on $N$. The constants $c_i$ are adjusted so that
$[\gamma_j]=[\alpha_j]=v_j= \varphi^*[\alpha_{j+1}] -[\alpha_{j+1}]$. So we can write
  $$
  \gamma_j= \varphi^*\alpha_{j+1} -\alpha_{j+1}- d\theta_{j+1}
  $$
for some $(p-1)$-form $\theta_{j+1}$ on $N$. Write
 $$
 \hat{\alpha}_{j+1}= \int_0^t \left( \tilde{\alpha}_j(x,s)- \sum_{i=1}^{j-1} c_i \tilde\alpha_i(x,s)\right) 
 ds + \alpha_{j+1} + d(\zeta(t)\theta_{j+1}(x))\,.
 $$
This is a $p$-form well-defined in $M$ since $\varphi^*_p(\hat{\alpha}_{j+1}(x,0))=
\varphi^*_p(\alpha_{j+1})=\gamma_j+\alpha_{j+1} + d\theta_{j+1} =
\hat{\alpha}_{j+1}(x,1)$. Set
 $$
 \tilde{\alpha}_{j+1}=\hat{\alpha}_{j+1} + \sum_{i<j} c_i \tilde{\alpha}_{i+1}
 $$
Then
 $$
 d\tilde{\alpha}_{j+1}=dt \wedge \tilde{\alpha}_{j} \, .
 $$
Finally,
 $$
 [\tilde{\alpha}_{j+1}|_{N\x\{t\}}] = [\alpha_{j+1}] + \sum_{i=1}^j c_i(t) [\alpha_i] \, ,
 $$
for some $c_i(t)$, as required.

Repeating this procedure with all Jordan blocks, we finally get
  $$
  \rho: (\bigwedge (W^1 \oplus W^p), d) \to (\Omega^*(M),d).
  $$
Clearly $H^p(\bigwedge (W^1 \oplus W^p))=K_1$, so $\rho^*$ is an
isomorphism on degree $p$. For degree $p+1$,
 $H^{p+1}(\bigwedge (W^1 \oplus W^p))$ is generated by the
 elements $a\cdot b$, where $b\in G_{j_0}$ corresponds
 to some $v\in K_{j_0}$ generating a Jordan block (equivalently,
 $v\notin I$). These elements generate $\coker (\varphi^*_{p}-\id)$, i.e.
  $$
  H^{p+1}(\bigwedge (W^1 \oplus W^p)) \cong \coker (\varphi^*_{p}-\id) .
  $$
An element $v=v_{j_0}$ is sent, by $\rho$, to a $p$-form
$\tilde\alpha_{j_0}$ on $M$, which satisfies
  $$
   [\tilde\alpha_{j_0}|_{N\x\{t\}}] = [\alpha_{j_0}] + \sum_{i=1}^{j_0-1} c_i [\alpha_i]\, ,
  $$
for some $c_i=c_i(t)$, following the previous notations. Therefore the class $[dt\wedge\tilde\alpha_{j_0}]$
corresponds to $[dt]\wedge [\alpha_{j_0}]$, in the notation of Lemma \ref{lem:18'}.
So
 $$
  \rho^*: H^{p+1}(\bigwedge (W^1 \oplus W^p)) \to H^{p+1}(M)
  $$
is the injection into the subspace $[dt]\wedge \coker (\varphi^*_{p}-\id)$.
This completes the proof of the theorem.
 \end{proof}

 Note that, in the notation of Proposition \ref{prop:criterio1}, we have that $C^1=W^1$, $C^p=G_1$
 and $N^p= G_2 \oplus \ldots \oplus G_r$.
 Also take $w\in G_r$. Then $a \cdot w \in I(N)$, $d(a\cdot w)=0$, but $a\cdot w$ is not
 exact. Hence

 \begin{corollary} \label{mapping-torus-p-formal}
 Under the conditions of Theorem \ref{thm:minimal-model-mapping-torus}, if $r\geq 2$ then $M$ is non-formal. Moreover, if $r=1$, then $M$ is $p$-formal (in the sense of Definition \ref{def:primera}).
  \end{corollary}

Applying this to symplectic mapping tori, we have the following.
Let $N$ be a compact symplectic $2n$-manifold, and assume that $\varphi:N\to N$ is a
symplectomorphism such that the map induced on cohomology
$\varphi^*_1: H^1(N) \to H^1(N)$ does not have the eigenvalue
$\lambda=1$. As $\varphi^*_2: H^2(N) \to H^2(N)$
always has the eigenvalue $\lambda=1$ ($\varphi^*$ fixes the symplectic form),
then we have that $N_{\varphi}$ is $2$-formal if and only if
the eigenvalue $\lambda=1$ of $\varphi_2^*$ has multiplicity $r=1$.

If $n=2$, then $N_\varphi$ is a $5$-dimensional co-symplectic manifold with $b_1=1$.
In dimension $5$, Theorem \ref{fm2:criterio2} says that $2$-formality is equivalent to formality.
Therefore we have the following result:

\begin{corollary}
 $5$-dimensional non-formal co-symplectic manifolds with $b_1=1$ are given
as mapping tori of symplectomorphisms $\varphi:N\to N$ of compact symplectic
$4$-manifolds $N$ where $ \varphi^*_1$ does not have the eigenvalue $\lambda=1$ and
$\varphi^*_2$ has the eigenvalue $\lambda=1$ with multiplicity $r\geq 2$.
 \end{corollary}

Finally, let us mention that an analogue of Theorem \ref{thm:minimal-model-mapping-torus}
for $p=1$ is harder to obtain. However, at least we can still say that if
$\lambda=1$ is an eigenvalue of $\varphi^*_1$ with
multiplicity $r\geq 2$, then $M=N_\varphi$ is non-formal (by Remark \ref{rem:otro}).
Also one can also obtain a {\em non-formal} mapping torus
such that $\lambda=1$ is an eigenvalue of $\varphi^*_1$ with multiplicity $r=1$, e.g. by taking
a non-formal symplectic nilmanifold $N$ and multiplying it by $S^1$.
Next, we give an example of a $5$-dimensional {\em formal} mapping torus
$N_{\varphi}$ with no co-symplectic structure and such that $\lambda=1$ is
an eigenvalue of $\varphi^*_1$ with multiplicity $r=1$.

Let $G(k)$ be the simply connected completely solvable\footnote{A solvable Lie group $G$ is \textit{completely solvable} if for every $X\in\fg$, the eigenvalues of the map $\textrm{ad}_X$ are real.} $3$-dimensional
Lie group defined by the equations
$$
de^1 = -k e^{1}\wedge e^{3}, \quad  de^2 = k e^{2}\wedge e^{3}, \quad de^3=0,
$$
where $k$ is a real number such that $\exp(k) + \exp(-k)$ is
an integer different from $2$. 

Let $\Gamma(k)$ be a discrete subgroup of $G(k)$ such that the quotient space
$P(k)=\Gamma(k)\backslash G(k)$ is compact (such a subgroup $\Gamma(k)$ 
always exists; see \cite{OT} for example). Then $P(k)$ is a completely solvable solvmanifold.

We can use Hattori's theorem \cite{Hattori} which asserts that the de Rham cohomology
ring $H^*(P(k))$ is isomorphic to the cohomology ring $H^*({\fg}^*)$ of the Lie algebra
$\fg$ of $G(k)$. For simplicity we denote the left invariant
forms $\{e^i\}$, $i=1,2,3$, on $G(k)$ and their projections on $P(k)$ by the same symbols.
Thus, we obtain
\begin{itemize}
\item $H^0(P(k))=\langle 1\rangle$,
\item $H^1(P(k))=\langle [e^3]\rangle$,
\item $H^2(P(k))=\langle [e^{12}]\rangle$,
\item $H^3(P(k))=\langle [e^{123}]\rangle$.
\end{itemize}

Therefore, there exists a real number $a$ such that the cohomology class
$a[e^{12}]$ is integral. Hence there exists a principal circle bundle
$\pi: N(k) \to P(k)$ with Euler class $a[e^{12}]$ and a connection $1$-form $e^4$
whose curvature form is $ae^{12}$ (we use the same notation for differential
forms on the base space $P(k)$ and their pullbacks via $\pi$ to the total space
$N(k)$).

One can check that the de Rham cohomology groups $H^*(N(k))$ are:
\begin{itemize}
\item $H^0(N(k))=\langle 1\rangle$,
\item $H^1(N(k))=\langle [e^3]\rangle$,
\item $H^2(N(k))=0$,
\item $H^3(N(k))=\langle [e^{124}]\rangle$,
\item $H^4(N(k))=\langle [e^{1234}]\rangle$.
\end{itemize}

Moreover, the manifold $N(k)$ is formal. In fact,
let $(\Omega^*(N(k)),d)$ be the
de Rham complex of differential forms on $N(k)$. The minimal
model of $N(k)$ is a differential graded algebra
$(\mathcal {M},d)$, with
 $$
 {\mathcal M}=\bigwedge(a,b),
 $$
where the generator $a$ has
degree $1$, the generator $b$ has degree 3, and $d$ is given
by $da=db=0$.
The morphism $\rho\colon {\mathcal M} \to \Omega^*(N(k))$, inducing
an isomorphism on cohomology, is defined by
 \begin{align*}
 \rho(a) &=e^3, \\
 \rho(b) &= e^{124}.
 \end{align*}
According to Definition \ref{def:primera}, we have $C^1=\la a\ra$
and $N^1=0$. Thus {\em $N(k)$ is $1$-formal} and hence
it is formal by Theorem \ref{fm2:criterio2}.

Now, let $M$ be the $5$-dimensional compact manifold defined as 
$M=N(k) \times S^1$. Denote by $e^5$
the canonical $1$-form on $S^1$. Then $M$ is formal.  Clearly 
$M$ is a mapping torus. But $M$ does not admit
co-symplectic structures since $H^2(M)=\langle [e^{35}]\rangle$, and so
any closed $2$-form $F$ satisfies $F^2=0$.

%
%
\section{Geography of non-formal co-symplectic compact manifolds} \label{sec:5}

In this section we consider the following problem:
\begin{quote}
  For which pairs $(m=2n+1,b)$, with $n,b \geq 1$,
 are there compact co-symplectic manifolds of dimension $m$ and with $b_1=b$
 which are non-formal?
\end{quote}

It will turn out that the answer is the same
as for compact smooth manifolds \cite{FM1}, i.e., that there are non-formal examples
if and only if
$m=3$ and $b \geq 2$, or $m \geq 5$ and $b \geq 1$. We start with some
straightforward examples:

\begin{itemize}
 \item For $b=1$ and $m\geq 9$, we may take
 a compact non-formal symplectic
 manifold $N$ of dimension $m-1\geq 8$ and simply-connected. Such manifold
 exists for dimensions $\geq 10$ by \cite{BT}, and for dimension equal to $8$ by \cite{FM3}. Then consider $M=N\times S^1$.
 \item For $m=3$, $b=2$, we may take the $3$-dimensional nilmanifold $M_{0}$ defined by the structure equations
 $de^1=de^2=0$, $de^3=e^1\wedge e^2$.
 This is non-formal since it is not a torus. The pair
 $\eta= e^1$, $F=e^2\wedge e^3$ defines a co-symplectic structure on $M_0$ since
 $d\eta =dF=0$ and $\eta\wedge F \not=0$.
 \item For $m\geq 5$ and $b\geq 2$ even, take the co-symplectic compact manifold
$M=M_{0}\times \Sigma_{k} \times (S^2)^{\ell}$, where $\Sigma_k$ is the surface of genus $k\geq 0$, $\ell\geq 0$,
and $(S^2)^{\ell}$ is the product of $\ell$ copies of $S^2$. Then
$\dim M=m=5+2\ell$ and $b_1(M)=2+2k$.
 \item For $m=5$ and $b=3$, we can take $M_1=N  \times S^1$, where $N$ is a compact $4$-dimensional symplectic manifold with
$b_1=2$. For example, take $N$ the compact nilmanifold defined by the equations $de^1=de^2=0$, $de^3=e^1\wedge e^2$,
$de^4=e^1\wedge e^3$, which is non-formal and symplectic with $\omega=e^1\wedge e^4 + e^2\wedge e^3$.
 \item For $m\geq 7$ and $b\geq 3$ odd, take $M=M_1\times \Sigma_k\times (S^2)^{\ell}$, $k,\ell\geq 0$.
\end{itemize}

Other examples with $b_1 = 2$ and $m=5$ can be obtained from the list of $5$-dimensional compact nilmanifolds.
According to the classification in \cite{BM, Mag} of nilpotent Lie algebras of dimension $<7$,
there are $9$ nilpotent Lie algebras $\mathfrak{g}$ of dimension $5$, and only $3$ of them
satisfy $\dim H^1({\mathfrak{g}}^*)=2$, namely
$$
(0,0,12,13,14+23), \quad (0,0,12,13,14), \quad (0,0,12,13,23).
$$

In the description of the Lie algebras $\fg$, we are using the structure equations
with respect to a basis $e^1,\ldots,e^5$ of the dual space $\fg^*$. For instance,
$(0,0,12,13,14+23)$ means that
there is a basis $\{e^j\}_{j=1}^5$ satisfying
$d e^1=d e^2=0$, $d e^3=e^{1}\wedge e^{2}$, $d e^4=e^{1}\wedge e^{3}$ and
$d e^5=e^{1}\wedge e^{4}+e^{2}\wedge e^{3}$; equivalently, the Lie bracket is
given in terms of its dual basis $\{e_j\}_{j=1}^5$ by
$[e_1,e_2]=-e_3$, $[e_1,e_3]=-e_4$, $[e_1,e_4]=[e_2,e_3]=-e_5$.
Also, from now on we write $e^{ij}=e^{i}\wedge e^{j}$.

\begin{proposition} \label{co-symplectic-nilmanifolds}
Among the $3$ nilpotent Lie algebras $\mathfrak{g}$ of dimension $5$
with $\dim H^1({\mathfrak{g}}^*)=2$, those that have a co-symplectic structure
are
 $$
 (0,0,12,13,14+23), \quad (0,0,12,13,14).
 $$
\end{proposition}

\begin{proof}
Clearly the forms $\eta$ and $F$ given by
$$
\eta= e^1,   \quad  F=e^{25}-e^{34}
$$
satisfy $d\eta =dF=0$ and $\eta\wedge F^2 \not=0$, and so
they define a co-sympectic structure on each of those Lie algebras.

To prove that the Lie algebra $(0,0,12,13,23)$
does not admit a co-sympectic structure, 
one can check it 
directly or use the fact that the direct sum of $(0,0,12,13,23)$ with the $1$-dimensional
Lie algebra has no symplectic form \cite{BM}.
\end{proof}

\begin{remark}\label{remark:4}
Let $N$ denote the $5$-dimensional compact
nilmanifold associated to the Lie algebra $\mathfrak{n}$ with structure
$(0,0,12,13,23)$. Then $N$ has a closed $1$-form; indeed, $de^1=de^2=0$. By
Tischler's theorem \cite{Tischler}, $N$ is a mapping torus. However, it is not
a \textit{symplectic} mapping torus, since it is not co-symplectic. We describe
this mapping torus explicitly. Since $N$ is a nilmanifold, we can describe
the structure at the level of Lie algebras.
The map $\mathfrak{n}\to\bR$, $(e_1,\ldots,e_5)\to e_1$ gives an exact sequence
\begin{equation}\label{extension}
0\longrightarrow \mathfrak{k}\longrightarrow\mathfrak{n}\longrightarrow\bR\longrightarrow 0
\end{equation}
of Lie algebras, and one sees immediately that $\mathfrak{k}$ is a $4$-dimensional symplectic
nilpotent Lie algebra, spanned by $e_2,\ldots,e_5$, with structure $(0,0,0,23)$
(with respect to the dual basis of $\mathfrak{k}^*$). The fiber of the corresponding bundle over $S^1$
is the Kodaira-Thurston manifold $KT$. Taking into account
the proof of Proposition \ref{prop:107}, the Lie algebra
extension (\ref{extension}) is associated to the derivation
$D=\mathrm{ad}(e_1)$ of $\mathfrak{k}$. In other words,
$\mathfrak{n}=\bR\oplus_D\mathfrak{k}$. A computation shows that
this derivation is not symplectic with respect to \textit{any} symplectic form on $\mathfrak{k}$ and Proposition \ref{prop:107} implies
that $\mathfrak{n}$ is not co-symplectic.
The map $\varphi:=\exp(D)$ is a diffeomorphism of $KT$ which does
not preserve any symplectic structure of $KT$, and $N=KT_{\varphi}$.
\end{remark}

The previous examples leave some gaps, notably the
cases $m=3$, $b\geq 3$, and $m=5$, $b=1$.
By \cite{FM1}, we know that there are compact non-formal manifolds with
these Betti numbers and dimensions. Let us see that there are also non-formal
co-symplectic manifolds in these cases.

\begin{proposition}\label{m=3}
  There are non-formal compact co-symplectic manifolds
  with $m \geq 3$, $b_1\geq 2$.
\end{proposition}

\begin{proof}
  We consider the symplectic surface $\Sigma_k$ of genus $k\geq 1$.
  Consider a symplectomorphism $\varphi:\Sigma_k\to \Sigma_k$ such that
  $\varphi^*:H^1(\Sigma_k) \to H^1(\Sigma_k)$ has the form
   $$
   \varphi^* = \left( \begin{array}{cc} 1 & 0 \\ 1 & 1\end{array} \right)
   \oplus \left( \begin{array}{cc} 1 & 0 \\ 0 & 1\end{array} \right)
   \oplus \ldots \oplus
   \left( \begin{array}{cc} 1 & 0 \\ 0 & 1\end{array} \right),
   $$
  with respect to a symplectic basis
 $\xi_1,\xi_2,\ldots, \xi_{2k-1},\xi_{2k}$ of $H_1(\Sigma_k)$.
  Consider the mapping torus $M$ of $\varphi$. The symplectic form
  of $\Sigma_k$ induces a closed $2$-form $F$ on $M$. The pull-back
  $\eta$ of the volume form of $S^1$ under $M\to S^1$ is closed and
  satisfies that  $\eta \wedge F >0$. Therefore $M$ is co-symplectic.

 Now $\varphi^{*}\xi_1 = \xi_1 + \xi_2$ and $\varphi^{*}\xi_i = \xi_i$,
 for $2\leq i\leq 2k$. By Lemma \ref{lem:18'}, the cohomology of $M$ is
   \begin{align*}
   H^1(M) &= \langle a, \xi_2,\ldots, \xi_{2k-1},\xi_{2k} \rangle, \\
   H^2(M) &= \langle F, a \, \xi_1, a\,  \xi_3, \ldots, a\,  \xi_{2k-1}, a\, \xi_{2k} \rangle,
   \end{align*}
   where $a=[\eta]$. So $b_1=2k\geq 2$.
   By  Theorem \ref{thm:19}, the Massey product
  $\la a,a, \xi_2\ra$ does not vanish and so $M$ is non-formal.

   Similarly, take $\Sigma_k$ where $k\geq 2$. We consider
  a symplectomorphism $\psi:\Sigma_k\to \Sigma_k$ such
  that $\psi^*:H^1(\Sigma_k) \to H^1(\Sigma_k)$ has the form
   $$
   \psi^* = \left( \begin{array}{cc} 1 & 0 \\ 1 & 1\end{array} \right)
   \oplus  \left( \begin{array}{cc} 1 & 0 \\ 1 & 1\end{array} \right)
   \oplus \left( \begin{array}{cc} 1 & 0 \\ 0 & 1\end{array} \right)
   \oplus \ldots \oplus
   \left( \begin{array}{cc} 1 & 0 \\ 0 & 1\end{array} \right) .
   $$
   Then the mapping torus $M$ of $\psi$ has $b_1=2k-1 \geq 3$ and odd, and
   $M$ is co-symplectic and non-formal.

   For higher dimensions, take $M\times (S^2)^{\ell}$, $\ell\geq 0$.
\end{proof}

\begin{remark}
 Notice that the case $k=1$ in the first part of the previous proposition yields
 another description of the Heisenberg manifold.
\end{remark}

\begin{proposition}\label{prop:3}
  There are non-formal compact co-symplectic manifolds
  with $m \geq 5$, $b_1= 1$.
\end{proposition}

\begin{proof}
It is enough to construct an example for $m=5$.
Take the torus $T^4$ and the mapping torus $T^4_{\varphi}$ of the
symplectomorphism $\varphi:T^4 \to T^4$ such that
\begin{equation}\label{eqn:eqn}
\varphi^* = \left( \begin{array}{cccc}
-1 & 0 & 0 & 0 \\
0 & -1 & 0 & 0 \\
0 & 0 &- 1 & 0 \\
0 & 0 & 1 & -1
\end{array} \right)
\end{equation}
on $H^1(T^4)$. Taking $\eta$ the pull-back of the $1$-form $\theta$ on $S^1$ and
$F=e^1\wedge e^2+e^3\wedge e^4$, we have that $T^4_{\varphi}$ is co-symplectic.
The map $\varphi^*$ on $H^2(T^4)$ satisfies:
\begin{align*}
\varphi^*(e^1\wedge e^2) &= e^1\wedge e^2 \\
\varphi^*(e^1\wedge e^3) &= e^1\wedge e^3 -e^1\wedge e^4 \\
\varphi^*(e^1\wedge e^4) &= e^1\wedge e^4 \\
\varphi^*(e^2\wedge e^3) &= e^2\wedge e^3 - e^2\wedge e^4 \\
\varphi^*(e^2\wedge e^4) &= e^2\wedge e^4 \\
\varphi^*(e^3\wedge e^4) &= e^3\wedge e^4
\end{align*}
Then $b_1(T^4_{\varphi})=1$ as $H^1(T^4_{\varphi}) = \langle a \rangle$,
with $a=[\eta]$. Also $H^2(T^4_{\varphi})= \langle e^{12}, e^{14}, e^{24}, e^{34}\rangle$.
In particular, notice that $\im(\varphi^*-\id)=\langle e^{14},e^{24}\rangle$.
Then $e^{14}\in\ker(\varphi^*-\id)$ and $e^{14}\in\im(\varphi^*-\id)$. So Theorem
\ref{thm:19} gives us the non-formality of $T^4_{\varphi}$.

For higher dimensions, take
$M=N\times (S^2)^{\ell}$, where $\ell\geq 0$. Then
$\dim M=5+2\ell$ and $b_1(M)=1$.
\end{proof}

\begin{remark}
Let us show that the $5$-manifold $T^4_{\varphi}$ is not a solvmanifold, that is,
it cannot be written as a quotient of a simply-connected solvable Lie group
by a discrete cocompact subgroup\footnote{If we
define solvmanifold as a quotient $\Gamma\backslash G$, where $G$
is a simply-connected solvable Lie
group and $\Gamma \subset G$ is a closed (not necessarily discrete) subgroup,
then any mapping torus $N_\varphi$, where $N$ is a nilmanifold is of this type
(see \cite{Mos}).}. The fiber bundle
 $$
 T^4\longrightarrow T^4_{\varphi}\longrightarrow S^1
 $$
gives a short exact sequence at the level of fundamental groups,
\begin{equation}\label{fundamental_groups}
0\longrightarrow \Z^4\longrightarrow H \longrightarrow \Z\longrightarrow 0,
\end{equation}
where $H=\pi_1(T^4_{\varphi})$. Since $\Z$ is free and $\Z^4$ is abelian,
one has $H=\Z\ltimes\Z^4$. Now suppose that
$T^4_{\varphi}$ is a solvmanifold of the form $\Gamma\backslash G$.
Clearly, it is $\Gamma \cong H$.
According to \cite{Mos}, we have a fibration
$$
N\longrightarrow T^4_{\varphi}\longrightarrow T^k
$$
where $N$ is a nilmanifold and $T^k$ is a $k$-torus. Since
$b_1(T^4_{\varphi})=1$, we have $k=1$ and $N$ is
a $4$-dimensional nilmanifold. This gives another short exact sequence of groups
$$
0\longrightarrow \Delta\longrightarrow \Gamma \longrightarrow \Z\longrightarrow 0,
$$
where $\Delta=\pi_1(N)$. 
But we know that there is a unique surjection
$H_{1}(\Gamma)=\Z\oplus T  \longrightarrow \Z$ (where $T$ is a torsion group)
and that, composed with the natural surjection 
$\Gamma  \longrightarrow {\Gamma}/{[\Gamma,\Gamma]} = H_{1}(\Gamma)$,
this gives a unique homomorphism $\Gamma  \longrightarrow \Z$.
Hence, the extension $ \Delta\longrightarrow \Gamma \longrightarrow \Z$
is the same as (\ref{fundamental_groups}).
Therefore $\Delta=\Z^4$. 
The Mostow fibration of
$\Gamma\backslash G=T^4_{\varphi}$ coincides with the mapping torus bundle.
At the level of Lie groups, it must be $G=\bR\ltimes\bR^4$ with semidirect product
$$
(t,x)\cdot(t',x')=(t+t',x+f(t)x')
$$
with $f$ a $1$-parameter subgroup in $\mathrm{GL}(4,\bR)$, i.e., $f(t)=\exp(tg)$ for
some matrix $g$.
Moreover, $f(1)=\exp(g)=\varphi^*$. But $\varphi^*$ can not be the exponential
of a matrix. Indeed, if $g$ has real eigenvalues, then $\varphi^*$ has
positive eigenvalues. If $g$ has purely imaginary eigenvalues and diagonalizes,
so does $\varphi^*$. And if $g$ has complex conjugate eigenvalues but does
not diagonalize, then $\varphi^*$ has two Jordan blocks.
None of these cases occur.
\end{remark}

\begin{remark}
The example constructed in the proof of Proposition \ref{prop:3} can
be used to give another example of a $5$-dimensional non-formal co-symplectic manifold
with $b_1=1$ which is not a solvmanifold.

Take $N=T^4$ and $\varphi:N \to N$ satisfying (\ref{eqn:eqn}). We may arrange
that $\varphi$ fixes the neighborhood of a point $p\in N$. Take the (symplectic)
blow-up of $N$ at $p$, $\widetilde{N} = N \# \overline{\mathbb{C}P} \hskip 0.05 cm^2$, and the 
induced symplectomorphism $\tilde\varphi:\widetilde{N}\to \widetilde{N}$. Let
$M=\widetilde{N}_{\tilde\varphi}$ be the corresponding mapping torus.
Clearly, $M$ is co-symplectic, it has $b_1(M)=1$ and the eigenvalue $\lambda=1$ 
of $\varphi^*:H^2(\widetilde{N}) \to H^2(\widetilde{N})$ has multiplicity $2$, hence
$M$ is non-formal. But $M$ cannot be a solvmanifold since $\pi_2(M)=\pi_2(\widetilde{N})=\Z$
is non-trivial.
\end{remark}

%
%

\section{A non-formal solvmanifold of dimension $5$ with $b_1=1$} \label{sec:6}

In this section we show an example of a non-formal compact
co-symplectic\footnote{Recall that the definition of co-symplectic 
manifold in this paper differs from that used in other papers, such as \cite{FV}.} $5$-dimensional
solvmanifold $S$ with first Betti number $b_1(S)=1$. Actually, $S$ is the mapping torus of
a certain diffeomorphism $\varphi$ of a $4$-torus preserving the orientation, so
this example fits in the scope of Proposition \ref{prop:3}.

Let $\fg$ be the abelian Lie algebra of dimension 4.
Suppose $\fg=\langle e_1,e_2,e_3,e_4\rangle$, and take the symplectic form
$\omega=e^{14}+e^{23}$ on $\fg$,
where $\langle e^1, e^2, e^3, e^4\rangle$
is the dual basis for the dual space ${\fg}^*$ such that the first cohomology group
$H^1({\fg}^*)=\langle [e^1], [e^2],[e^3],[e^4]\rangle$.
Consider the endomorphism of $\fg$ represented, with respect to the chosen basis, by the matrix
$$
D=\begin{pmatrix}
-1 & 0 & 0 & 0\\ 0 & 1 & 0 & 0\\-1 & 0 & -1 & 0\\ 0 & -1 & 0 & 1
\end{pmatrix}.
$$
It is immediate to see that $D$ is an infinitesimal symplectic transformation.
Since $\fg$ is abelian, it is also a derivation. Applying Proposition \ref{prop:107}
we obtain a co-symplectic Lie algebra
$$
\fh=\bR\xi\oplus \fg
$$
with brackets defined by
$$
[\xi,e_1]=-e_1-e_3, \quad [\xi,e_2]=e_2-e_4, \quad [\xi,e_3]=-e_3 \quad\textrm{and}\quad [\xi,e_4]=e_4.
$$

One can check that $\fh=\langle e_1,e_2,e_3,e_4,e_5=\xi\rangle$ is
a completely solvable non-nilpotent Lie algebra. We denote
by $\langle \alpha_1, \alpha_2, \alpha_3, \alpha_4, \alpha_5\rangle$
the dual basis for $\mathfrak{h}^*$.
The Chevalley-Eilenberg complex of $\mathfrak{h}^*$ is
$$
(\bigwedge(\alpha_1,\ldots,\alpha_5),d)
$$
with differential $d$ defined by
\begin{align*}
 d\alpha_1 &=-\alpha_1\wedge\alpha_5, \\
 d\alpha_2 &=\alpha_2\wedge\alpha_5, \\
 d\alpha_3 &=-\alpha_1\wedge\alpha_5-\alpha_3\wedge\alpha_5, \\
 d\alpha_4 &=-\alpha_2\wedge\alpha_5+\alpha_4\wedge\alpha_5, \\
 d\alpha_5 &=0.
\end{align*}

Let $H$ be the simply connected and completely solvable Lie group of dimension $5$
consisting of matrices of the form
 $$
 a=\begin{pmatrix}
 e^{-{x_5}}&0&0&0&0&x_1 \\
 0&e^{x_5}&0&0&0&x_2\\
 -{x_5}e^{-{x_5}}&0&e^{-{x_5}}&0&0&x_3\\
 0&-x_{5}e^{x_{5}}&0&e^{x_{5}}&0&x_4\\
 0&0&0&0&1&{x_5} \\
 0&0&0&0&0&1
 \end{pmatrix},
 $$
where $x_i \in \bR$, for $1\leq i \leq 5$. Then a global system of
coordinates $\{x_i, 1\leq i\leq 5\}$ for $H$ is defined by
 $x_i(a)=x_i$,  and a standard
calculation shows that a basis for the left invariant $1$-forms
on $H$ consists of
 $$
 \alpha_1=e^{x_5}dx_1, \quad \alpha_2=e^{-{x_5}}dx_2, \quad \alpha_3={x_5}e^{x_{5}}dx_1+e^{x_{5}}dx_3,
 \quad \alpha_4= {x_5}e^{-{x_5}}dx_2+e^{-{x_5}}dx_4, \quad \alpha_5=dx_5.
 $$

This means that $\fh$ is the Lie algebra of $H$.
 We notice that the Lie group $H$ may be described as a semidirect product
 $H=\bR \ltimes_{\rho} {\bR}^4$, where $\bR$ acts on ${\bR}^4$ via the linear transformation
 $\rho(t)$ of  ${\bR}^4$ given by the matrix
 $$
\rho(t)=\begin{pmatrix}
 e^{-t}&0&0&0\\
 0&e^{t}&0&0\\
 -{t}e^{-t}&0&e^{-t}&0\\
 0&-{t}e^{t}&0&e^{t}
 \end{pmatrix}.
 $$
Thus the operation on the group $H$ is given by
  $$
  \mathbf{a}\cdot\mathbf{x}=(a_1+x_1e^{-{a_5}},a_2+x_2 e^{a_5},a_3+x_3e^{-a_5}-a_5x_1e^{-{a_5}},a_4+x_4e^{a_5}-a_5x_2e^{a_5},a_5+x_5).
  $$
 where $\mathbf{a}=(a_1,\ldots,a_5)$ and similarly for $\mathbf{x}$. Therefore
 $H=\bR \ltimes_{\rho} {\bR}^4$, where $\bR$ is a connected abelian subgroup, and
 ${\bR}^4$ is the nilpotent commutator subgroup.

  Now we show that there exists a discrete subgroup $\Gamma$ of $H$ such that
  the quotient space
  $\Gamma\backslash H$ is compact. To construct $\Gamma$
  it suffices to find some real number $t_0$ such that the matrix defining
  $\rho(t_0)$ is conjugate to an element $A$ of the special linear group
  $\mathrm{SL}(4,\Z)$ with distinct real eigenvalues $\lambda$ and ${\lambda}^{-1}$.
  Indeed, we could then find a lattice $\Gamma_0$ in  ${\bR}^4$ which is
  invariant under ${\rho(t_0)}$, and take
  $\Gamma=(t_0 {\Z})\ltimes_{\rho}\Gamma_0$. To this end,
  we choose the matrix $A\in\mathrm{SL}(4,\Z)$
  given by
  \begin{equation}\label{eqn:A-matrix}
    A=\begin{pmatrix}
 2&1&0&0 \\
 1&1&0&0\\
2&1&2&1\\
 1&1&1&1
 \end{pmatrix},
  \end{equation}
with double eigenvalues $\frac{3+\sqrt{5}}{2}$ and $\frac{3-\sqrt{5}}{2}$. Taking
  $t_0=\log(\frac{3+\sqrt{5}}{2})$, we have that the matrices $\rho(t_0)$
  and $A$ are conjugate. Indeed, put
    \begin{equation}\label{eqn:P-matrix}
   P=\begin{pmatrix}
 1&\frac{-2(2+\sqrt{5})}{3+\sqrt{5}}&0&0 \\
 1&\frac{1+\sqrt{5}}{3+\sqrt{5}}&0&0\\
0&0&\log(\frac{2}{3+\sqrt{5}})&\frac{2(2+\sqrt{5})\log(\frac{3+\sqrt{5}}{2})}{3+\sqrt{5}}\\
 0&0&\log(\frac{2}{3+\sqrt{5}})&-\frac{(1+\sqrt{5})\log(\frac{3+\sqrt{5}}{2})}{3+\sqrt{5}}
 \end{pmatrix},
  \end{equation}
 then a direct calculation shows that $PA=\rho(t_0)P$.
 So the lattice $\Gamma_0$ in  ${\bR}^4$ defined by
 $$
 \Gamma_0 = P(m_1, m_2, m_3, m_4)^{t},
 $$
where $m_1, m_2, m_3, m_4 \in \Z$ and  $(m_1, m_2, m_3, m_4)^{t}$ is the transpose
of the vector   $(m_1, m_2, m_3, m_4)$, is invariant under the subgroup $t_0 {\Z}$.
Thus $\Gamma=(t_0 {\Z})\ltimes_{\rho}\Gamma_0$ is a cocompact subgroup of $H$.

We denote by  $S=\Gamma\backslash H$ the compact quotient manifold.
Then $S$ is a $5$-dimensional (non-nilpotent) completely solvable solvmanifold.

Alternatively, $S$ may be viewed as the total space of a
$T^4$-bundle over the circle $S^1$. In fact, let $T^4 = \Gamma_0 \backslash {\bR}^4$ be the
$4$-dimensional torus and $\varphi \colon {\Z} \rightarrow \mathrm{Diff}(T^4)$ the representation
defined as follows: $\varphi(m)$ is the transformation of $T^4$ covered by the linear
transformation of ${\bR}^4$ given by the matrix
$$
\rho(mt_0)= \begin{pmatrix}
 e^{-mt_{0}}&0&0&0\\
 0&e^{mt_{0}}&0&0\\
 -mt_{0}e^{-mt_{0}}&0&e^{-mt_{0}}&0\\
 0&{-mt_{0}}e^{mt_{0}}&0&e^{mt_{0}}
 \end{pmatrix}.
$$
So $\Z$ acts on $T^4\times \bR$ by
 $$
 ((x_1,x_2,x_3,x_4),x_5) \mapsto (\rho(mt_0)\cdot(x_1,x_2,x_3,x_4)^t, x_5+m),
 $$
 and $S$ is the quotient $(T^4\times \bR)/\Z$. The projection $\pi$ is given by
 $$
 \pi[(x_1, x_2, x_3, x_4), x_5] = [x_5].
 $$

\begin{remark} \label{rem:solvmanifold-is-mapping-torus}
We notice that
$S$ is a mapping torus associated to a certain symplectomorphism
$\Phi:T^4\to T^4$. Indeed, since $D$ is an infinitesimal symplectic transformation,
its exponential $\exp(tD)$ is a $1$-parameter group of symplectomorphisms of $\bR^4$.
Notice that $\exp(tD)=\rho(t)$. We saw that there exists a number $t_0\in\bR$
such that $\rho(t_0)$ preserves a lattice $\Gamma_0\cong\Z^4\subset\bR^4$.
Therefore the symplectomorphism $\rho(t_0)$ descends to a symplectomorphism
$\Phi$ of the $4$-torus $\Gamma_0\backslash\bR^4$, whose mapping torus is precisely $\Gamma\backslash H$.
\end{remark}

Next, we compute the real cohomology of $S$. Since $S$ is completely solvable, Hattori's theorem \cite{Hattori} says that the de Rham cohomology
ring $H^*(S)$ is isomorphic to the cohomology ring $H^*({\fh}^*)$ of the Lie algebra
$\fh$ of $H$. For simplicity we denote the left invariant
forms $\{\alpha_i\}$, $i=1,\ldots, 5$, on $H$ and their projections on $S$ by the same symbols.
Thus, we obtain

\begin{itemize}
\item $H^0(S)=\langle 1\rangle$,
\item $H^1(S)=\langle [\alpha_5]\rangle$,
\item $H^2(S)=\langle [\alpha_1\wedge\alpha_2], [\alpha_1\wedge\alpha_4+\alpha_2\wedge\alpha_3]\rangle$,
\item $H^3(S)=\langle [\alpha_3\wedge\alpha_4\wedge\alpha_5], [(\alpha_1\wedge\alpha_4+\alpha_2\wedge\alpha_3)\wedge\alpha_5\rangle$,
\item $H^4(S)=\langle [\alpha_1\wedge\alpha_2\wedge\alpha_3\wedge\alpha_4] \rangle$,
\item $H^5(S)=\langle [\alpha_1\wedge\alpha_2\wedge\alpha_3\wedge\alpha_4\wedge \alpha_5]\rangle$.
\end{itemize}

The product $H^1(S)\otimes H^2(S)\to H^3(S)$ is given by
 $$
  [\alpha_1\wedge\alpha_4+\alpha_2\wedge\alpha_3]\wedge  [\alpha_5] =
  [(\alpha_1\wedge\alpha_4+\alpha_2\wedge\alpha_3)\wedge\alpha_5] \quad \text{and}
  \quad  [\alpha_1\wedge\alpha_2] \wedge [\alpha_5]=0.
  $$

\begin{theorem}\label{thm:upgraded-4}
$S$ is a compact co-symplectic $5$-manifold which is non-formal and
with first Betti number $b_1(S)=1$.
\end{theorem}

\begin{proof}
Take the 1-form $\eta=\alpha_5$, and let $F$ be the 2-form on $S$ given by
$$
F= \alpha_1\wedge\alpha_4+\alpha_2\wedge\alpha_3.
$$
Then $(F, \eta)$ defines a co-symplectic structure on $S$ since $dF=d\eta=0$ and
$\eta\wedge F^2\not=0$.

We prove the non-formality of $S$ from its minimal model \cite{OT}.
The minimal model of $S$ is a differential graded algebra
$(\mathcal {M},d)$, with
 $$
 {\mathcal M}=\bigwedge(a) \otimes \bigwedge(b_1,b_2,b_3,b_4) \otimes \bigwedge V^{\geq 3},
 $$
where the generator $a$ has
degree $1$, the generators $b_i$ have degree 2, and $d$ is given
by $da=db_1=db_2=0$, $db_3= a\cdot b_2$, $db_4=a\cdot b_3$.
The morphism $\rho\colon {\mathcal M} \to \Omega^*(S)$, inducing
an isomorphism on cohomology, is defined by
 \begin{align*}
 \rho(a) &=\alpha_5, \\
 \rho(b_1) &= \alpha_1\wedge\alpha_4+\alpha_2\wedge\alpha_3,\\
 \rho(b_2) &=\alpha_1 \wedge \alpha_2,\\
 \rho(b_3) &=\frac12 (\alpha_1\wedge \alpha_4 - \alpha_2\wedge\alpha_3),\\
 \rho(b_4) &= \frac12 \alpha_{3}\wedge \alpha_4.
 \end{align*}

Following the notations in Definition \ref{def:primera}, we have $C^1=\la a\ra$
and $N^1=0$, thus {\em $S$ is $1$-formal}. We see that {\em $S$ is
not $2$-formal}. In fact, the element $b_4 \cdot a \in N^2 \cdot
V^1$ is closed but not exact, which implies that $({\mathcal M},d)$ is
not $2$-formal. Therefore, $({\mathcal M},d)$ is not formal.
\end{proof}

\begin{remark}
It can be seen that $S$ is non-formal by computing a \emph{quadruple}
Massey product \cite{OT}
$\la [\alpha_1\wedge\alpha_2], [\alpha_5], [\alpha_5],[\alpha_5] \ra$.
As $\alpha_1\wedge\alpha_2\wedge\alpha_5= \frac12 d(\alpha_1\wedge \alpha_4 - \alpha_2\wedge\alpha_3)$
and $(\alpha_1\wedge \alpha_4 - \alpha_2\wedge\alpha_3)\wedge\alpha_5=d(\alpha_{3}\wedge \alpha_4)$,
we have
 $$
 \la [\alpha_1\wedge\alpha_2], [\alpha_5], [\alpha_5],[\alpha_5] \ra
 = \frac12 [\alpha_3\wedge\alpha_4\wedge\alpha_5].
 $$
This is easily seen to be non-zero modulo the indeterminacies.
\end{remark}

\begin{remark}
Theorem \ref{thm:upgraded-4} can be also proved
with the techniques of section \ref{sec:5}. By
Remark \ref{rem:solvmanifold-is-mapping-torus},
$S$ is the mapping torus of a diffeomorphism $\rho(t_0)$ of
$T^4=\Gamma_0\backslash \bR^4$.
Conjugating by the matrix $P$ in (\ref{eqn:P-matrix}), we have that $S$ is the
mapping torus of $A$ in (\ref{eqn:A-matrix}) acting on the standard $4$-torus
$T^4=\Z^4\backslash \bR^4$. The action of $A$ on $1$-forms leaves
no invariant forms, so $b_1(S)=1$. The action of $A$ on $2$-forms is
given by the matrix
 $$
 \begin{pmatrix} 1 & 0 &0& 0& 0& 0 \\
 0& 4& 2& 2& 1& 0 \\ 1 & 2& 2& 1& 1& 0 \\ -1& 2& 1& 2& 1&
    0 \\ 0& 1& 1& 1& 1& 0 \\ 1& 0& 1& -1& 0& 1 \end{pmatrix},
 $$
with respect to the basis $\{e^{12},e^{13}, e^{14}, e^{23}, e^{24}, e^{34} \}$.
This matrix has eigenvalues $\lambda= \frac12 (7 \pm 3\sqrt{5})$ (simple) and
$\lambda=1$, with multiplicity $3$ (one block of size $1$ and another of size $3$).
Theorem \ref{thm:minimal-model-mapping-torus} implies the
non-formality of $S$.
\end{remark}

\begin{remark}
We notice that the previous example $S$ may be generalized to
dimension $2n+1$ with $n\geq 2$. For this, it is enough to consider
the $(2n+1)$-dimensional completely solvable Lie group
$H^{2n+1}$ defined by the structure equations
\begin{itemize}
\item $d\alpha_j=(-1)^j\alpha_j\wedge\alpha_{2n+1}$, $j=1,\ldots,2n-2$;
\item $d\alpha_{2n-1}=-\alpha_1\wedge\alpha_{2n+1}-\alpha_{2n-1}\wedge\alpha_{2n+1}$;
\item $d\alpha_{2n}=-\alpha_2\wedge\alpha_{2n+1}+\alpha_{2n}\wedge\alpha_{2n+1}$;
\item $d\alpha_{2n+1}=0$.
\end{itemize}
The co-symplectic structure $(\eta, F)$ is defined by $\eta=\alpha_{2n+1}$,
and
$F= \alpha_1\wedge\alpha_{2n}+\alpha_2\wedge\alpha_{2n-1}+\alpha_3\wedge\alpha_4
+\cdots+\alpha_{2n-3}\wedge\alpha_{2n-2}$.
\end{remark}

\bigskip

\vskip.3cm

\noindent {\bf Acknowledgments.} We are very grateful to the referee for useful comments
that helped to improve the paper. The first and third authors were partially supported
by Project MICINN (Spain) MTM2010-17389. The second author was partially
supported through Project MICINN (Spain) MTM2011-28326-C02-02, and Project
of UPV/EHU ref.\ UFI11/52.

\smallskip

{\small

\end{document}